\documentclass[12pt,final]{article}
\usepackage[margin=2.5cm,top=1.5cm]{geometry}
\usepackage[all,arc]{xy}
\usepackage{amsmath,amsthm,amssymb,enumerate}
\usepackage{color}
\newtheorem{thm}{Theorem}%[section]
\newtheorem{con}[thm]{Conjecture}%[section]
\newtheorem{lem}[thm]{Lemma}%[section]
\theoremstyle{definition}
\theoremstyle{Remark}
\newtheorem{rem}[thm]{Remark}%[section]

\newcommand{\Romannum}[1]{\uppercase\expandafter{\romannumeral #1}}

\numberwithin{equation}{section}
\newcommand\keywordsname{Key words}
\newcommand\AMSname{AMS subject classifications}

\newenvironment{@abssec}[1]{%
     \if@twocolumn
       \section*{#1}%
     \else
       \vspace{.05in}\footnotesize
       \parindent .2in
         {\upshape\bfseries #1. }\ignorespaces
     \fi}
     {\if@twocolumn\else\par\vspace{.1in}\fi}

\begin{document}
\title{\vspace*{3cm} The Minimal Total Irregularity of Graphs\footnote{Research supported by the Zhujiang Technology New Star Foundation of
Guangzhou (No.2011J2200090), and Program on International Cooperation and Innovation, Department of Education,
Guangdong Province (No.2012gjhz0007).}}
\author{Yingxue Zhu\footnote{781722521@qq.com}\qquad
Lihua You\footnote{{\it{Corresponding author:\;}}ylhua@scnu.edu.cn}
\quad Jieshan Yang\footnote{jieshanyang1989@163.com}}
\vskip.2cm
\date{{\small
School of Mathematical Sciences, South China Normal University, Guangzhou, 510631, P.R. China\\
}} \maketitle
\vspace{-7mm}
\begin{abstract}
In \cite{2012a}, Abdo and Dimitov defined the total irregularity of a graph $G=(V,E)$ as

\hskip3.3cm $\rm irr_{t}$$(G) = \frac{1}{2}\sum_{u,v\in V}|d_{G}(u)-d_{G}(v)|,  $

\noindent where $d_{G}(u)$
denotes the vertex degree of a vertex $u\in V$.
In this paper, we investigate the minimal total irregularity of the connected graphs,
determine the minimal, the second minimal, the third minimal total irregularity  of trees, unicyclic graphs, bicyclic graphs on $n$ vertices,
and propose an open problem for further research.
%\vskip.2cm \noindent{\it{AMS classification:}} 05C50; 15A09; 15A48
\end{abstract}

\baselineskip=0.30in
\section{Introduction}
\hskip.6cm Let $G=(V,E)$ be a simple undirected graph with vertex set $V$ and edge set $E$.
For any  vertices $v\in V$, the degree of a vertex $v$ in $G$, denoted by $d_{G}(v)$,
 is the number of edges of $G$ incident with $v$.
If $V=\{v_1,v_2,\ldots,v_n\}$, then the sequence $(d_G(v_1), d_G(v_2),\ldots, d_G(v_n))$ is called a degree sequence of $G$ (\cite{1976}).
Without loss of generality, we assume $d_G(v_1)\geq d_G(v_2)\geq\ldots\geq d_G(v_n)$.

A graph is regular if all its vertices have the same degree, otherwise it is irregular. Several approaches that characterize how irregular a graph is have been proposed. In \cite{1997}, Alberson defined the imbalance of an edge $e=uv\in E$ as $|d_{G}(u)-d_{G}(v)|$ and the irregularity of $G$ as

\vskip.2cm
\hskip5cm $\rm irr$$(G)=\sum\limits_{uv\in E}|d_{G}(u)-d_{G}(v)|.$ \hskip5cm $(1)$

\noindent More results on the imbalance, the irregularity of a graph $G$ can be found in  \cite{1997}-\cite{2005}.

Inspired by the structure and meaning of the equation (1), Abdo and Dimitov \cite{2012a} introduced a new irregularity measure, called the total irregularity. For a graph $G$, it is defined as

\vskip.2cm
\hskip4.8cm $\rm irr_{t}$$(G)=\frac{1}{2}\sum\limits_{u,v\in V}|d_{G}(u)-d_{G}(v)|.$ \hskip4.8cm $(2)$

Although the two irregularity measures capture the irregularity only by a single parameter, namely the degree of a vertex, the new measure is more superior than the old one in some aspects. For example, (2) has an expected property of an irregularity measure that graphs with the same degree sequences have the same total irregularity, while (1) does not have.
Both measures also have common properties, including that they are zero if and only if $G$ is regular.

Obviously, $\rm irr_{t}$$(G)$ is an upper bound of $\rm irr$$(G)$. In  \cite{2012c}, the authors derived relation between $\rm irr_{t}$$(G)$ and $\rm irr$$(G)$ for a connected graph $G$ with  $n$ vertices, that is,
$\rm irr_{t}$$(G)\leq n^2$$ \rm irr$$(G)/4.$
 Furthermore, they showed that $\rm irr_{t}$$(T)\leq (n-2)$$\rm irr$$(T)$ for any tree $T$.

Let $P_{n}$, $C_{n}$ and $S_{n}$ be the path, cycle and  star on $n$ vertices, respectively.
In \cite{2012a}, the authors obtained the upper bound of the  total irregularity among all graphs on $n$ vertices,
 and they showed the star graph $S_n$ is the tree with the maximal total irregularity among all trees on $n$ vertices.

\begin{thm}{\rm (\cite{2012a})} Let $G$ be a simple, undirected graph on $n$ vertices. Then

{\rm (1) }  $\rm irr_{t}$$(G)\leq \frac{1}{12}(2n^3-3n^2-2n+3).$

{\rm (2) } If $G$ is a tree, then $\rm irr_{t}$$(G)\leq (n-1)(n-2), $ with equality  holds if and only if $G\cong S_n$.
\end{thm}

In \cite{2013}, the authors investigated  the total irregularity of  unicyclic graphs,
and determined the graph with the maximal total irregularity $n^2-n-6$ among all unicyclic graphs on $n$ vertices.
In \cite{2014}, the authors  investigated  the total irregularity of bicyclic graphs, and determined the graph with the maximal total irregularity $n^2+n-16$ among all bicyclic graphs on $n$ vertices.

Recently,   Abdo and Dimitrov (\cite{2013a}) also obtained the upper bounds on the total irregularity of graphs under several graph operations including join, lexicographic product, Cartesian product, strong product, direct product, corona product, disjunction and symmetric difference and so on.

In this paper, we introduce an important transformation to investigate the minimal total irregularity of graphs in Section 2,
determine the minimal, the second minimal, the third minimal total irregularity  of trees, unicyclic graphs, bicyclic graphs on $n$ vertices in Sections 3-5,
and propose an open problem for further research.

\section{Branch-transformation}
\hskip.6cm In this section, we introduce an important transformation to investigate the minimal total irregularity of  graphs on $n$ vertices.

Let $G$ be a  graph on $n$ vertices, $T$ be an induced subtree of $G$. We call $T$ is a hanging tree of $G$ if $G$ can be formed by connecting a vertex of $T$ and a vertex of $G-T$.

\noindent {\bf Branch-transformation: }
Let $G$ be a simple  graph with at least two pendent vertices. Without loss of generality, let $u$ be a vertex of $G$ with $d_{G}(u)\geq3$, $T$ be a hanging tree of $G$ connecting to $u$ with $|V(T)|\geq1$, and $v$ be a pendent vertex of $G$ with $v\notin T$. Let $G^{\prime}$ be the graph obtained from $G$ by deleting $T$ from vertex $u$ and attaching it to vertex $v$. We call the transformation from $G$ to $G^{\prime}$ is a branch-transformation on $G$ from vertex $u$ to vertex $v$ (see Figure 1).

\hskip.5cm
\begin{picture}(500,80)

\put(70,50){\circle*{2}}
\put(90,50){\circle*{2}}
\put(130,50){\circle*{2}}
\put(150,50){\circle*{2}}
\put(110,50){\circle{40}}
\put(166,50){\circle{30}}
\put(70,50){\line(1,0){20}}
\put(130,50){\line(1,0){20}}
\put(105,45){$G_{0}$}
\put(65,53){$v$}
\put(131,53){$u$}
\put(160,45){$T$}

\put(195,50){$\Longrightarrow$}

\put(260,50){\circle*{2}}
\put(280,50){\circle*{2}}
\put(300,50){\circle*{2}}
\put(340,50){\circle*{2}}
\put(320,50){\circle{40}}
\put(244,50){\circle{30}}
\put(260,50){\line(1,0){20}}
\put(280,50){\line(1,0){20}}
\put(315,45){$G_{0}$}
\put(280,53){$v$}
\put(341,53){$u$}
\put(240,45){$T$}
\put(110,15){$G$}
\put(280,15){$G^{\prime}$}

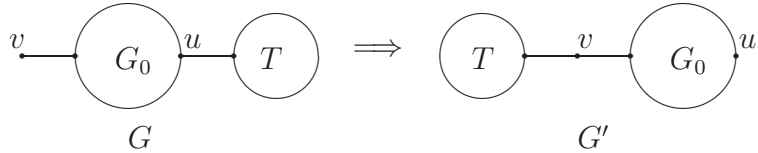
\put(70,0){Figure 1. branch-transformation on $G$ from $u$ to $v$}

\end{picture}

\begin{lem}\label{lem2}
Let $G^{\prime}$ be the graph obtained from $G$ by branch-transformation from $u$ to $v$. Then $\rm irr_{t}$$(G)>\rm irr_{t}$$(G^{\prime}).$
\end{lem}

\begin{proof}
Let $G=(V,E)$, $V_{1}=$\{$w| d_{G}(w)\geq d_{G}(u),w\in V$\}, $V_{2}=$\{$w| d_{G}(w)=1, w\in V$\}, $V_{3}=\{w| 2\leq d_{G}(w)<d_{G}(u), w\in V\}.$ Clearly, $u\in V_{1}$, $v\in V_{2}$, and $V_{1}\cup V_{2}\cup V_{3}=V$. Let $|V_{1}|=s$, $|V_{2}|=h$, $|V_{3}|=r$, then $s\geq 1$, $h\geq 2$ and  $s+h+r=n$.

Note that after  banch-transformation, only the degrees of $u$ and $v$
have been changed, namely, $d_{G^{\prime}}(u)=d_{G}(u)-1$, $d_{G^{\prime}}(v)=d_{G}(v)+1=2$ and $d_{G^{\prime}}(x)=d_{G}(x)$ for any $x\in V\backslash\{u,v\}$.
Let $U= V\backslash\{u,v\}$. Then

\hskip3.8cm $|d_{G^{\prime}}(u)-d_{G^{\prime}}(v)|-|d_{G}(u)-d_{G}(v)|=-2,$

 \hskip0.6cm $\sum\limits_{x\in U}(|d_{G^{\prime}}(u)-d_{G^{\prime}}(x)|
-|d_{G}(u)-d_{G}(x)|)=(s-1)-(r+h-1)=s-r-h,$

\hskip0.6cm  $\sum\limits_{x\in U}(|d_{G^{\prime}}(v)-d_{G^{\prime}}(x)|
-|d_{G}(v)-d_{G}(x)|)=-(s-1)-r+(h-1)=-s-r+h.$

 Thus, we have

$\rm irr_{t}$$(G^{\prime})-$$\rm irr_{t}$$(G)$

\noindent \hskip.2cm$=|d_{G^{\prime}}(u)-d_{G^{\prime}}(v)|
+\sum\limits_{x\in U}|d_{G^{\prime}}(u)-d_{G^{\prime}}(x)|
+\sum\limits_{x\in U}|d_{G^{\prime}}(v)-d_{G^{\prime}}(x)|$

\noindent \hskip.2cm $-(|d_{G}(u)-d_{G}(v)|
+\sum\limits_{x\in U}|d_{G}(u)-d_{G}(x)|
+\sum\limits_{x\in U}|d_{G}(v)-d_{G}(x)|)$

\noindent \hskip.2cm $=(|d_{G^{\prime}}(u)-d_{G^{\prime}}(v)|-|d_{G}(u)-d_{G}(v)|)$
 $+\sum\limits_{x\in U}(|d_{G^{\prime}}(u)-d_{G^{\prime}}(x)|-|d_{G}(u)-d_{G}(x)|)$

\noindent \hskip.2cm $+\sum\limits_{x\in U}(|d_{G^{\prime}}(v)-d_{G^{\prime}}(x)|-|d_{G}(v)-d_{G}(x)|)$

\noindent \hskip.2cm $=-2+(s-r-h)+(-s-r+h)$

\noindent \hskip.2cm $=-2r-2$

\noindent \hskip.2cm $<0$.
\end{proof}

\begin{rem}\label{rem3}
 Let $G^{\prime}$ be the graph obtained from $G$ by branch-transformation from $u$ to $v$. Then by branch-transformation and Lemma \ref{lem2},
we have $d_{G^{\prime}}(u)=d_{G}(u)-1\geq 2$ and $d_{G^{\prime}}(v)=d_{G}(v)+1=2$, namely,
$|\{w|d_{G^{\prime}}(w)=1, w\in V\}|=|\{w|d_G(w)=1, w\in V\}|-1$.
If $d_{G^{\prime}}(u)\geq 3$,  $G^{\prime}$ has at least two pendent vertices, and there exists a hanging tree of $G^{\prime}$ connecting to the vertex $u$,
 we can repeat branch-transformation on $G^{\prime}$ from the vertex $u$,  till the degree of $u$ is equal to 2, or there is only one pendent vertex in the resulting graph,  or there is not any hanging tree  connecting to the  vertex $u$.
\end{rem}

From the above arguments, we see that we can do branch-transformation on $G$  if and only if the following three conditions hold:

{\rm (1) } there exists a vertex $u$ with $d_G(u)\geq 3$;

{\rm (2) } there exists a hanging tree of $G$ connecting to the vertex $u$;

{\rm (3) } $G$ has at least two pendent vertices.

\section{The minimal total irregularity of trees}
\hskip.6cm In this section, we determine the minimal, the second minimal, the third minimal total irregularity of trees on $n$ vertices
and characterize the extremal graphs.

\begin{lem}\label{lem4}{\rm (\cite{1976})}
Let $G=(V, E)$ be a graph and $|E|=m$. Then $\sum\limits_{v\in V}d_G(v)=2m$.
\end{lem}

Let $G=(V,E)$ be a tree. Then for any vertex $u\in V$,  $d_G(u)\geq 2$ implies there must exist a hanging tree of $G$ connecting to the vertex $u$,
thus we can obtain the following results by branch-transformation.

\begin{thm}\label{thm5}
Let $G=(V,E)$ be a tree on $n$ vertices. Then $\rm irr_{t}$$(G)\geq 2n-4$,
and the equality holds if and only if $G\cong P_{n}$.
\end{thm}

\begin{proof}
Clearly, $2(n-1)=\sum\limits_{v\in V}d_G(v)$ by Lemma \ref{lem4}.
 Let $s=|\{w|d_{G}(w)\geq 3, w\in V\}|$, and $h=|\{w|d_{G}(w)=1, w\in V\}|$.
 Then  $s\geq 0$ and $h\geq 2$. Let $\triangle{(G)}$ be the maximum degree of the vertices of $G$.
  Now  we complete the proof by the following two cases.

\noindent {\bf Case 1: }  $s=0$.

Then $h=2$ by $2(n-1)=\sum\limits_{v\in V}d_G(v)=2(n-h)+h$, and the degree sequence of $G$ is $(2,\ldots,2,1,1)$.
Thus $G\cong P_{n}$ and $\rm irr_{t}$$(G)=2n-4$.

\noindent {\bf Case 2: }  $s\geq 1$.

Then  $\triangle(G)\geq 3$ by $s\geq 1$, and $h\geq \triangle(G)+s-1\geq 3$ by $2(n-1)=\sum\limits_{v\in V}d_G(v)\geq \triangle(G)+3(s-1)+2(n-s-h)+h$.
So we can do branch-transformation $h-2$ times  on $G$  till the degree sequence of the resulting  graph  is $(2, \ldots, 2, 1, 1)$, denoted by $H_1$,
and thus   $\rm irr_{t}$$(G)>$$\rm irr_t$$(H_1)=2n-4$ by Lemma \ref{lem2}.
\end{proof}

\begin{thm}\label{thm6}
Let $n\geq 5$, $G=(V,E)$ be a tree on $n$ vertices and $G \ncong P_{n}$. Then $\rm irr_{t}$$(G)\geq 4n-10$,
and the equality holds if and only if  the degree sequence of $G$ is $(3, 2, \ldots, 2, 1, 1, 1)$.
\end{thm}

\begin{proof}
It is obvious that $2(n-1)=\sum\limits_{v\in V}d_G(v)$ by Lemma \ref{lem4}.
 Let $s=|\{w|d_{G}(w)\geq 3, w\in V\}|$, and $h=|\{w|d_{G}(w)=1, w\in V\}|$. Then  $\triangle{(G)}\geq 3$ and $s\geq1$ since $G \ncong P_{n}$.
 Now  we complete the proof by the following two cases.

\noindent {\bf Case 1: }  $s+\triangle(G)=4$.

 Clearly, $s=1, \triangle(G)=3$. Then $h=3$ by $2(n-1)=\sum\limits_{v\in V}d_G(v)=3+2(n-1-h)+h$, and the degree sequence of $G$ is $(3,2,\ldots,2,1,1,1)$.
Thus $\rm irr_{t}$$(G)=4n-10$.

\noindent {\bf Case 2: }  $s+\triangle(G)\geq 5$.

 Then $h\geq \triangle(G)+s-1\geq4$ by $2(n-1)=\sum\limits_{v\in V}d_G(v)\geq \triangle(G)+3(s-1)+2(n-s-h)+h$.
Now we can do branch-transformation $h-3$ times on $G$ till the  degree sequence of the resulting graph is $(3,2,\ldots,2,1,1,1)$,
denoted by $H_2$, and thus   $\rm irr_{t}$$(G)>$$\rm irr_t$$(H_2)=4n-10$ by Lemma \ref{lem2}.
\end{proof}

\begin{thm}\label{thm7}
Let $n\geq 6$, $G=(V,E)$ be a tree on $n$ vertices, $G\ncong P_n.$ If the sequence $(3,2,\ldots,2,1,1,1)$ is not the degree sequence of $G,$
then  $\rm irr_{t}$$(G)\geq 6n-20$,
and the equality holds if and only if  the degree sequence of $G$ is $(3, 3, 2, \ldots, 2, 1, 1, 1, 1)$.
\end{thm}

\begin{proof}
Clearly,  $2(n-1)=\sum\limits_{v\in V}d_G(v)$ by Lemma \ref{lem4}. Let $s=|\{w|d_{G}(w)\geq 3, w\in V\}|$, and $h=|\{w|d_{G}(w)=1, w\in V\}|$.
Then $s\geq1$  and $\triangle{(G)}\geq 3$ since  $G \ncong P_{n}$. Now   we complete the proof by the following two cases.

\noindent {\bf Case 1: }  $s=1$.

Then $\triangle(G)\geq 4$ because  sequence $(3,2,\ldots,2,1,1,1)$ is not the degree sequence of $G$.

\noindent {\bf Subcase 1.1: }  $\triangle(G)= 4$.

Then $h=4$ by $2(n-1)=\sum\limits_{v\in V}d_G(v)=4+2(n-1-h)+h$, and the degree sequence of $G$ is $(4,2,\ldots,2,1,1,1,1)$.
Thus $\rm irr_{t}$$(G)=6n-18>6n-20$.

\noindent {\bf Subcase 1.2: }   $\triangle(G)\geq 5$.

Then $h=\triangle(G)\geq 5$ by $2(n-1)=\sum\limits_{v\in V}d_G(v)= \triangle(G)+2(n-1-h)+h$.
Now we can do branch-transformation $h-4$ times on $G$ till the  degree sequence of the resulting graph  is $(4,2,\ldots,2,1,1,1,1)$,
denoted by $H_3$, and thus   $\rm irr_{t}$$(G)>$$\rm irr_t$$(H_3)=6n-18>6n-20$ by Lemma \ref{lem2}.

\noindent {\bf Case 2: }  $s\geq 2$.

\noindent {\bf Case 2.1: }  $s+\triangle(G)=5$.

Then  the degree sequence of $G$ is $(3,3,2,\ldots,2,1,1,1,1)$, and thus $\rm irr_{t}$$(G)=6n-20$.

\noindent {\bf Case 2.2: }  $s+\triangle(G)\geq 6$.

 Then $h\geq\triangle(G)+s-1\geq 5$ by $2(n-1)=\sum\limits_{v\in V}d_G(v)\geq \triangle(G)+3(s-1)+2(n-s-h)+h$.
Now we can do branch-transformation $h-4$ times on $G$ till the  degree sequence of the resulting graph is $(3,3,2,\ldots,2,1,1,1,1)$,
denoted by $H_4$, and thus   $\rm irr_{t}$$(G)>$$\rm irr_t$$(H_4)=6n-20$ by Lemma \ref{lem2}.
\end{proof}

\begin{rem}\label{rem8}
Let $n\geq 6$, by Theorems \ref{thm5}-\ref{thm7}, we know the minimal, the second minimal, the third minimal total irregularity of trees on $n$ vertices
are $2n-4$, $4n-10$, $6n-20$, respectively, and the degree sequences of the corresponding extremal graphs are $(2,\ldots, 2,1,1)$,
$(3,2,\ldots,2,1,1,1)$, $(3,3,2,\ldots,2,1,1,1,1)$, respectively.
\end{rem}

\section{The minimal total irregularity of unicyclic graphs}
\hskip.6cm In this section, we determine the minimal, the second minimal, the third minimal total irregularity of unicyclic graphs on $n$ vertices
and characterize the extremal graphs.

An unicyclic graph is a simple connected graph in which the number of edges equals the number of vertices.
Let $G=(V,E)$ be an unicyclic graph. Then for any vertex $u\in V$,  $d_G(u)\geq 3$ implies there must exist a hanging tree of $G$ connecting to the vertex $u$,
thus we can obtain the following results by branch-transformation.

\begin{thm}\label{thm9}
Let $n\geq 3$ and $G=(V,E)$ be an unicyclic graph on $n$ vertices.

{\rm (1) }  $\rm irr_{t}$$(G)\geq 0 $,  and the equality holds if and only if $G\cong C_{n}$.

{\rm (2) }  Let $n\geq 4$, and $G\ncong C_n$. Then $\rm irr_{t}$$(G) \geq 2n-2$,  and
the equality holds if and only if the degree sequence of $G$ is $(3,2,\ldots, 2,1)$. % {\rm (see Figure 1)}.
\end{thm}

\begin{proof}
  {\rm (1) } is obvious. Now we show {\rm (2) } holds.

It is obvious that $2n=\sum\limits_{v\in V}d_G(v)$ by Lemma \ref{lem4}.
Let $s=|\{w|d_{G}(w)\geq 3, w\in V\}|$, and $h=|\{w|d_{G}(w)=1, w\in V\}|$.
Then $s\geq1$, $h\geq 1$, $\triangle(G)\geq 3$ by $G \ncong C_{n}$  and $2n=\sum\limits_{v\in V}d_G(v)$.
Now we complete the proof by the following two cases.

\noindent {\bf Case 1: }  $s+\triangle(G)=4$.

Then $(3,2,\ldots, 2,1)$ is the degree sequence of $G$ by $2n=\sum\limits_{v\in V}d_G(v)=3+2(n-1-h)+h$, and
thus $\rm irr_{t}$$(G)=2n-2$.

\noindent {\bf Case 2: }   $s+\triangle(G)\geq 5$.

Then $h\geq \triangle(G)+s-3\geq2$ by  $2n=\sum\limits_{v\in V}d_G(v)\geq \triangle(G)+3(s-1)+2(n-s-h)+h$,
and we can do branch-transformation $h-1$ times on $G$ till the  degree sequence of the resulting graph  is $(3,2,\ldots,2,1)$, denoted by $H_5$,
and thus   $\rm irr_{t}$$(G)>$$\rm irr_t$$(H_5)=2n-2$ by Lemma \ref{lem2}.
\end{proof}

 \begin{thm}\label{thm10} %µÚÈýС
Let $n\geq 5$,  $G=(V,E)$ be an unicyclic graph on $n$ vertices with $G\ncong C_n.$ If the sequence  $(3,2,\ldots, 2,1)$ is not  the degree sequence of $G,$ then $\rm irr_{t}$$(G)\geq 4n-8$,
 and the equality holds if and only if the degree sequence of $G$ is $(3,3, 2,\ldots, 2, 1, 1)$.
\end{thm}

\begin{proof}
Clearly, $2n=\sum\limits_{v\in V}d_G(v)$ by Lemma \ref{lem4}.
Let $s=|\{w|d_{G}(w)\geq 3, w\in V\}|$, and $h=|\{w|d_{G}(w)=1, w\in V\}|$.
Then $s\geq1$, $h\geq 1$ by $G \ncong C_{n}$  and $2n=\sum\limits_{v\in V}d_G(v)$.
Now we complete the proof by the following two cases.

 \noindent {\bf Case 1: }  $s=1$.

Then $\triangle(G)\geq 4$ because  sequence $(3,2,\ldots,2,1)$ is not the degree sequence of $G$.

\noindent {\bf Subcase 1.1: }  $\triangle(G)=4$.

Then $h=2$ and the degree sequence is $(4,2,\ldots,2,1,1)$ by $2n=\sum\limits_{v\in V}d_G(v)=4+2(n-1-h)+h$,
 and thus $\rm irr_{t}$$(G)=4n-6>4n-8$.

\noindent {\bf Subcase 1.2: }    $\triangle(G)\geq5$.

Then $h=\triangle(G)-2\geq 3$ by $2n=\sum\limits_{v\in V}d_G(v)=\triangle(G)+2(n-1-h)+h$,
and we can do branch-transformation $h-2$ times on $G$ till the  degree sequence of the resulting graph is $(4,2,\ldots,2,1,1)$,
denoted by $H_6$, and thus   $\rm irr_{t}$$(G)>$$\rm irr_t$$(H_6)=4n-6$ by Lemma \ref{lem2}.

\noindent {\bf Case 2: }  $s\geq2$.

\noindent {\bf Subcase 2.1: }   $s+\triangle(G)=5$.

Then  $h=2$ and the degree sequence is $(3,3,2,\ldots,2,1,1)$ by $2n=\sum\limits_{v\in V}d(v)=6+2(n-2-h)+h$, and
thus  $\rm irr_{t}$$(G)=4n-8$.

\noindent {\bf Subcase 2.2: }   $s+\triangle(G)\geq 6$.

Then $h\geq \triangle(G)+s-3\geq 3$ by $2n=\sum\limits_{v\in V}d_G(v)\geq \triangle(G)+3(s-1)+2(n-s-h)+h$,
and we can do branch-transformation $h-2$ times on $G$ till the  degree sequence of the resulting graph  is $(3,3,2,\ldots,2,1,1)$, denoted by $H_7$,
and thus   $\rm irr_{t}$$(G)>$$\rm irr_t$$(H_7)=4n-8$ by Lemma \ref{lem2}.
\end{proof}

\begin{rem}\label{rem11}
Let $n\geq 5$, by Theorems \ref{thm9}-\ref{thm10}, we know the minimal, the second minimal, the third minimal total irregularity of unicyclic graphs on $n$ vertices are $0$, $2n-2$, $4n-8$, respectively, and the degree sequences of the corresponding extremal graphs are $(2,\ldots, 2)$,
$(3,2,\ldots,2,1)$, $(3,3,2,\ldots,2,1,1)$, respectively.
\end{rem}

\section{The minimal total irregularity of bicyclic graphs}
\hskip.6cm In this section, we determine the minimal, the second minimal, the third minimal total irregularity of bicyclic graphs on $n$ vertices
and characterize the extremal graphs.

A bicyclic graph is a simple connected graph in which the number of edges equals the number of vertices plus one.
There are two basic bicyclic graphs: $\infty$-graph and $\Theta$-graph. An $\infty$-graph, denoted by
$\infty(p, q, l)$ (see Figure 2), is obtained from two vertex-disjoint cycles $C_p$ and $C_q$ by connecting one vertex
of $C_p$ and one of $C_q$ with a path $P_l$ of length $l-1$
(in the case of $l=1$, identifying the above two vertices, see Figure 3) where $p,q\geq 3$ and $l\geq 1$;
and a $\Theta$-graph, denoted by $\theta(p, q, l)$ (see Figure 4),
is a graph  on $p+q-l$ vertices with the two cycles $C_p$ and $C_q$ have $l$ common vertices, where $p,q\geq 3$ and $l\geq 2$.

\vspace{0.3cm}
\hskip-.4cm
\begin{picture}(400, 50)

\put(180,25){\circle{40}}
\put(200,25){\circle*{2}}
\put(200,25){\line(1,0){20}}
\put(220,25){\circle*{2}}
\put(223,22){$\cdots$}
\put(240,25){\circle*{2}}
\put(240,25){\line(1,0){20}}
\put(260,25){\circle*{2}}
\put(280,25){\circle{40}}

\put(175,23) {\small $C_p$}
\put(275,23) {\small $C_q$}
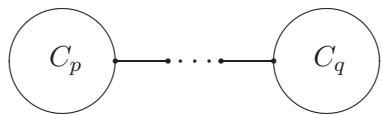
\put(130,-20) {\small Figure 2. The graph $\infty(p,q,l)$ with $l\geq 2$}
\end{picture}\\

\vspace{0.2cm}
\hskip-.4cm
\begin{picture}(400, 50)
\put(210,20){\circle{40}}
\put(250,20){\circle{40}}

\put(230,20){\circle*{2}}
%\put(210,0){\circle*{2}}
%\put(258,1){\circle*{2}}

\put(207,15) {\small $C_p$}
\put(250,15) {\small $C_q$}
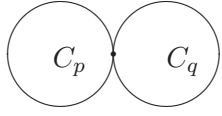
\put(160,-27) {\small Figure 3. The graph $\infty(p,q,1)$}
\end{picture}\\
\vspace{0.3cm}

\hskip-.4cm
\vspace{0.3cm}
\begin{picture}(500,110)

\put(200,50){\circle*{2}}
\put(220,50){\circle*{2}}
\put(240,50){\circle*{2}}
\put(260,50){\circle*{2}}
\put(200,90){\circle*{2}}
\put(220,90){\circle*{2}}
\put(240,90){\circle*{2}}
\put(260,90){\circle*{2}}
\put(200,10){\circle*{2}}
\put(220,10){\circle*{2}}
\put(240,10){\circle*{2}}
\put(260,10){\circle*{2}}

\put(225,10){\circle*{1}}
\put(230,10){\circle*{1}}
\put(235,10){\circle*{1}}

\put(225,50){\circle*{1}}
\put(230,50){\circle*{1}}
\put(235,50){\circle*{1}}

\put(225,90){\circle*{1}}
\put(230,90){\circle*{1}}
\put(235,90){\circle*{1}}

\put(200,50){\line(1,0){20}}
\put(200,90){\line(1,0){20}}
\put(200,10){\line(1,0){20}}

\put(240,50){\line(1,0){20}}
\put(240,90){\line(1,0){20}}
\put(240,10){\line(1,0){20}}

\put(200,50){\line(0,-1){40}}
\put(200,90){\line(0,-1){40}}

\put(260,50){\line(0,-1){40}}
\put(260,90){\line(0,-1){40}}

\put(190,53){\footnotesize$z_1$}
\put(215,53){\footnotesize$z_2$}
%\put(142,53){\footnotesize$z_{l-1}$}
\put(265,53){\footnotesize$z_{l}$}

\put(190,94){\footnotesize$x_1$}
\put(215,94){\footnotesize$x_2$}
%\put(142,94){\footnotesize$x_{p-l-1}$}
\put(265,94){\footnotesize$x_{p-l}$}

\put(190,2){\footnotesize$y_1$}
\put(215,2){\footnotesize$y_2$}
%\put(142,2){\footnotesize$y_{q-l-1}$}
\put(265,2){\footnotesize$y_{q-l}$}

\put(225,70){$C_p$}
\put(225,25){$C_q$}

\put(165,-25) {\small Figure 4. The graph $\theta(p,q,l)$}
\end{picture}
\vspace{0.2cm}

Denoted by $\mathcal{B}_n$ is the set of all bicyclic graphs on $n$ vertices.
Obviously, $\mathcal{B}_n$ consists of three types of graphs:
first type denoted by $B^{+}_n$, is the set of those graphs each of which is an $\infty$-graph, $\infty(p,q,l)$,  with trees attached when $l=1$;
second type denoted by $B^{++}_n$, is the set of those graphs each of which is an $\infty$-graph, $\infty(p,q,l)$,  with trees attached when $l\geq 2$;
third type denoted by $\Theta_n$, is the set of those graphs each of which is a $\Theta$-graph, $\theta(p,q,l)$,  with trees attached.
Then $\mathcal{B}_n=B_n^+\cup B_n^{++}\cup \Theta_n.$

\subsection{The graph with minimal total irregularity in $B^{+}_n$}
\hskip.6cm In this subsection,  the minimal, the second minimal total irregularity of the bicyclic graphs in $B^{+}_n$ are determined.

\begin{thm}\label{thm12}
Let $n\geq 6$,  $G=(V,E)\in B^{+}_n$.

{\rm (1) } $\rm irr_t$$(G)\geq 2n-2$, and the equality holds if and only if  the degree sequence of $G$ is $(4,2,\ldots,2)$.

{\rm (2) } If $(4,2,\ldots,2)$  is not the degree sequence of $G$, then $\rm irr_{t}$$(G)\geq 4n-6$, and the equality holds if and only if
 the degree sequence of $G$ is $(4,3,2,\ldots,2,1)$.
 \end{thm}

\begin{proof}
Clearly, $\sum\limits_{v\in V}d_G(v)=2(n+1)$ by Lemma \ref{lem4}.
Let $s=|\{w|d_{G}(w)\geq 3, w\in V\}|$, $h=|\{w|d_{G}(w)=1, w\in V\}|$ and $t=|\{w|d_{G}(w)=\triangle(G), w\in V\}|$.
Then $s\geq1$, $h\geq 0$, $1\leq t\leq s$ and $\triangle(G)\geq4$ by   $G\in B^{+}_n$.

Note that  $G\in B^{+}_n$, if $s=1$, $\triangle(G)\geq 5$ or $s\geq 2$, there must exist a vertex $u$ with $d_G(u)\geq 3$ and
there exists a hanging tree of $G$ connecting to $u$.  Then we complete the proof by the following two cases.

\noindent {\bf Case 1: }  $s=1$.

\noindent {\bf Subcase 1.1: }  $\triangle(G)=4$.

Then $h=0$ and the degree sequence of $G$ is $(4,2,\ldots, 2)$  by the fact  $2(n+1)=\sum\limits_{v\in V}d_G(v)=4+2(n-1-h)+h$,
and thus $\rm irr_{t}$$(G)=2n-2$.

\noindent {\bf Subcase 1.2: }    $\triangle(G)=5$.

Then $h=1$ and the degree sequence of $G$ is $(5,2,\ldots, 2, 1)$  by the fact  $2(n+1)=\sum\limits_{v\in V}d_G(v)=5+2(n-1-h)+h$,
and thus $\rm irr_{t}$$(G)=4n-4>4n-6$.

\noindent {\bf Subcase 1.3: }   $\triangle(G)\geq6$.

Then $h=\triangle(G)-4\geq2$ by the fact  $2(n+1)=\sum\limits_{v\in V}d_G(v)=\triangle(G)+2(n-1-h)+h$,
and we can do branch-transformation $h-1$ times on $G$ till the  degree sequence of the resulting graph  is $(5,2,\ldots,2,1)$,
 denoted by $H_8$, and thus   $\rm irr_{t}$$(G)>$$\rm irr_t$$(H_8)=4n-4$ by Lemma \ref{lem2}.

\noindent {\bf Case 2: }  $s\geq2$.

\noindent {\bf Subcase 2.1: }   $s+\triangle(G)=6$.

Then $s=2$, $\triangle(G)=4$ and $1\leq t\leq 2$.

If $t=1$, then $h=1$ and the degree sequence of $G$ is $(4,3, 2,\ldots, 2, 1)$  by the fact  $2(n+1)=\sum\limits_{v\in V}d_G(v)=4+3+2(n-2-h)+h$,
and thus $\rm irr_{t}$$(G)=4n-6$.

If $t=2$, then $h=2$  by the fact  $2(n+1)=\sum\limits_{v\in V}d_G(v)=4+4+2(n-2-h)+h$,
and we can do branch-transformation once on $G$ such that the  degree sequence of the resulting graph is $(4,3,2\ldots,2,1)$,
 denoted by $H_9$, and thus   $\rm irr_{t}$$(G)>$$\rm irr_t$$(H_9)=4n-6$ by Lemma \ref{lem2}.

\noindent {\bf Subcase 2.2: }   $s+\triangle(G)\geq 7$.

Then $h\geq \triangle(G)+s-5\geq 2$   by the fact  $2(n+1)=\sum\limits_{v\in V}d_G(v)\geq \triangle(G)+3(s-1)+2(n-s-h)+h$,
and we can do branch-transformation $h-1$ times on $G$ such that the  degree sequence of the resulting graph  is $(4,3,2\ldots,2,1)$,
 denoted by $H_9$, and thus   $\rm irr_{t}$$(G)>$$\rm irr_t$$(H_9)=4n-6$ by Lemma \ref{lem2}.
\end{proof}

\subsection{The graph with minimal total irregularity in $B^{++}_n$}
\hskip.6cm In this subsection,  the minimal, the second minimal total irregularity of the bicyclic graphs in $B^{++}_n$ are determined.

\begin{thm}\label{thm13}
Let $n\geq 7$,  $G=(V,E)\in B^{++}_n$.

{\rm (1) } $\rm irr_t$$(G)\geq 2n-4$, and the equality holds if and only if  the degree sequence of $G$ is $(3,3,2,\ldots,2)$.

{\rm (2) } If $(3,3,2,\ldots,2)$  is not the degree sequence of $G$, then $\rm irr_{t}$$(G)\geq 4n-10$, and the equality holds if and only if
 the degree sequence of $G$ is $(3,3,3,2,\ldots,2,1)$.
\end{thm}

\begin{proof}
Clearly, $\sum\limits_{v\in V}d_G(v)=2(n+1)$ by Lemma \ref{lem4}.
Let $s=|\{w|d_{G}(w)\geq 3, w\in V\}|$, $h=|\{w|d_{G}(w)=1, w\in V\}|$ and $t=|\{w|d_{G}(w)=\triangle(G), w\in V\}|$.
Then $s\geq2$, $h\geq 0$, $1\leq t\leq s$ and $\triangle(G)\geq3$ by $G\in B^{++}_n$.

 Note that  $G\in B^{++}_n$,  if $s=2$, $\triangle(G)\geq 4$ or $s\geq 3$, there must exist a vertex $u$ with $d_G(u)\geq 3$ and
there exists a hanging tree of $G$ connecting to $u$.  Then we complete the proof by the following two cases.

\noindent {\bf Case 1: }  $s=2$.

\noindent {\bf Subcase 1.1: } $\triangle(G)=3$.

Then $h=0$ and the degree sequence of $G$ is $(3,3,2,\ldots, 2)$ by the fact $2(n+1)=\sum\limits_{v\in V}d_G(v)=3+3+2(n-2-h)+h$,
and thus $\rm irr_{t}$$(G)=2n-4$.

\noindent {\bf Subcase 1.2: }  $\triangle(G)=4$.

Then $t=1$ or $t=2$ by $1\leq t\leq s$.

If $t=1$, then $h=1$ and the degree sequence of $G$ is $(4,3,2,\ldots, 2,1)$ by the fact $2(n+1)=\sum\limits_{v\in V}d_G(v)=4+3+2(n-2-h)+h$,
and thus $\rm irr_{t}$$(G)=4n-6>4n-10$.

If $t=2$, then  $h=2$  by the fact $2(n+1)=\sum\limits_{v\in V}d_G(v)=4+4+2(n-2-h)+h$,
and we can do branch-transformation once on $G$ such that the  degree sequence of the resulting graph  is $(4,3,2\ldots,2,1)$,
denoted by $H_{10}$, and thus   $\rm irr_{t}$$(G)>$$\rm irr_t$$(H_{10})=4n-6$ by Lemma \ref{lem2}.

\noindent {\bf Subcase 1.3: }   $\triangle(G)\geq 5$.

Then $h\geq \triangle(G)-3\geq 2$  by the fact $2(n+1)=\sum\limits_{v\in V}d_G(v)\geq\triangle(G)+3+2(n-2-h)+h$,
and we can do branch-transformation $h-1$ times on $G$ such that the  degree sequence of the resulting graph  is $(4,3,2\ldots,2,1)$,
denoted by $H_{10}$, and thus   $\rm irr_{t}$$(G)>$$\rm irr_t$$(H_{10})=4n-6$ by Lemma \ref{lem2}.

\noindent {\bf Case 2: }  $s\geq 3$.

\noindent {\bf Subcase 2.1: }  $s+\triangle(G)=6$.

Then $h=1$ and the degree sequence of $G$ is $(3,3,3,2,\ldots, 2,1)$ by the fact $2(n+1)=\sum\limits_{v\in V}d_G(v)=3+3+3+2(n-3-h)+h$,
and thus $\rm irr_{t}$$(G)=4n-10$.

\noindent {\bf Subcase 2.2: }  $s+\triangle(G)\geq 7$.

Then $h\geq \triangle(G)+s-5\geq 2$  by the fact $2(n+1)=\sum\limits_{v\in V}d_G(v)\geq\triangle(G)+3(s-1)+2(n-s-h)+h$,
and we can do branch-transformation $h-1$ times on $G$ such that the  degree sequence of the resulting graph  is $(3,3,3,2\ldots,2,1)$,
denoted by $H_{11}$, and thus   $\rm irr_{t}$$(G)>$$\rm irr_t$$(H_{11})=4n-10$ by Lemma \ref{lem2}.
\end{proof}

\subsection{The graph with minimal total irregularity in $\Theta_n$}
\hskip.6cm By the same proof of Theorem \ref{thm13}, we can  determine the minimal,
the second minimal total irregularity of the bicyclic graphs in $\Theta_n$ immediately.

\begin{thm}\label{thm14}
Let $n\geq 5$,  $G=(V,E)\in \Theta_n$.

{\rm (1) } $\rm irr_t$$(G)\geq 2n-4$, and the equality holds if and only if  the degree sequence of $G$ is $(3,3,2,\ldots,2)$.

{\rm (2) } If $(3,3,2,\ldots,2)$  is not the degree sequence of $G$, then $\rm irr_{t}$$(G)\geq 4n-10$, and the equality holds if and only if
 the degree sequence of $G$ is $(3,3,3,2,\ldots,2,1)$.
\end{thm}

\subsection{The graph with minimal total irregularity in $\mathcal{B}_n$}
\hskip.6cm By Theorems \ref{thm12}-\ref{thm14}, we can determine   the minimal, the second minimal,
the third minimal total irregularity of the bicyclic graphs  on $n$ vertices immediately.

\begin{thm}\label{thm15}
Let $n\geq 7$, $G\in \mathcal{B}_n$.

{\rm (1) } $\rm irr_t$$(G)\geq 2n-4$, and the equality holds if and only if  the degree sequence of $G$ is $(3,3,2,\ldots,2)$.

{\rm (2) } If $(3,3,2,\ldots,2)$  is not the degree sequence of $G$, then $\rm irr_t$$(G)\geq 2n-2$, and the equality holds if and only if  the degree sequence of $G$ is $(4,2,\ldots,2)$.

{\rm (3) } If $(3,3,2,\ldots,2)$ and $(4,2,\ldots,2)$ are not the degree sequence of $G$, then $\rm irr_{t}$$(G)\geq 4n-10$, and the equality holds if and only if
 the degree sequence of $G$ is $(3,3,3,2,\ldots,2,1)$.
\end{thm}

\section{Open problem for further research}
\hskip.6cm By the results of Sections 3-5, we know the minimal irregularity of simple connected graphs on $n$ vertices is zero,
and the corresponding extremal graphs are  regular graphs.
Furthermore, we suppose $2n-4$ is the second minimal and $2n-2$ is the third minimal.
\begin{con}\label{cor16}
Let $G$ be a simple connected graph with $n$ vertices. If $G$ is not a regular graph, then $\rm irr_{t}$$(G)\geq 2n-4$.
\end{con}

\end{document}